\documentclass[12pt,letterpaper,reqno]{amsart}
\PassOptionsToPackage{dvipsnames}{xcolor}
\usepackage{tikz}
\usetikzlibrary{positioning, shapes.geometric, arrows.meta, calc, backgrounds}
\usepackage{amssymb}
\usepackage{amsmath}
\usepackage{amsthm}
\usepackage{amsfonts}
\IfFileExists{bbm.sty}{\usepackage{bbm}}{}
\usepackage{enumitem}
\usepackage{pgfplots}
\pgfplotsset{compat=1.18}
\usepackage{booktabs,tabularx,array}
\usepackage{graphicx}
\usepackage[T1]{fontenc}
\usepackage{doi}
\usepackage{float}

\usepackage[dvipsnames]{xcolor}
\usepackage{hyperref}
\usepackage{bookmark}
\hypersetup{hypertexnames=false,
colorlinks=true,
linkcolor=RoyalBlue,
citecolor=ForestGreen!65!white,
urlcolor=BrickRed}
\allowdisplaybreaks

\newtheorem{thm}{Theorem}[section]
\newtheorem{lem}[thm]{Lemma}
\newtheorem{prop}[thm]{Proposition}
\newtheorem{cor}[thm]{Corollary}

\newtheorem{conj}[thm]{Conjecture}
\theoremstyle{definition}

\numberwithin{equation}{section}

\newcommand{\dd}{\,\mathrm d}
\newcommand{\e}{\mathrm e}
\newcommand{\Ree}{\operatorname{Re}}
\newcommand{\code}[1]{\texttt{\detokenize{#1}}}
\newcommand{\repoURL}{https://github.com/zhangteng2000/quadratic-tang-zhang-conjecture}
\newcommand{\releaseURL}{https://github.com/zhangteng2000/quadratic-tang-zhang-conjecture/releases/tag/v1.0}
\makeatother

\begin{document}
\title[The quadratic Tang--Zhang inequality]
{Beyond Sendov's conjecture: the quadratic Tang--Zhang inequality}
\author[T.~Zhang]{Teng Zhang}
\address{School of Mathematics and Statistics, Xi'an Jiaotong University, Xi'an 710049, P. R. China}
\email{teng.zhang@stu.xjtu.edu.cn}
\date{September 16, 2026}
\begin{abstract} Very recently, Lech Mazur proved the celebrated Sendov conjecture, and Terence Tao subsequently distilled the main ideas of the proof in a blog post. In this paper, we establish a quantitative strengthening of Sendov's conjecture, namely the quadratic Tang--Zhang inequality. Let $p$ be a polynomial of degree $n\ge2$ whose zeros lie in the closed unit disk, and let $\zeta_1,\ldots,\zeta_{n-1}$ denote its critical points, counted with multiplicity. We prove that, for every zero $a$ of $p$,
$$
\sum_{j=1}^{n-1}\frac{1}{|a-\zeta_j|^2}\ge n-1.
$$
Moreover, equality holds if and only if
$p(z)=c(z^n-\omega)$
for some $c\in\mathbb C\setminus\{0\}$ and $|\omega|=1$. We also provide a Lean 4 formalization of the main results.
\end{abstract}
\maketitle
\enlargethispage{2pt}

\section{Introduction}\label{sec:intro}

In 1958, Sendov formulated the following
conjecture, now known as \emph{Sendov's conjecture}; see his account in \cite[p.~284]{Sen02}. It was subsequently
included in Hayman's influential 1967 monograph
\emph{Research Problems in Function Theory}~\cite[Problem~4.5]{Hay67}.
    \begin{conj}[Sendov]\label{conj:Sendov}
Let $p$ be a polynomial of degree $n\ge2$ with all its zeros in the closed unit disk. For every zero $a$ of $p$, there is a critical point $\zeta$ such that
\begin{equation*}
    |a-\zeta|\le 1.
\end{equation*}
\end{conj}

Over the past several decades, this beautiful conjecture has attracted considerable attention. The conjecture was gradually verified in low degrees: Rubinstein verified the conjecture for $n=3,4$ \cite{Rub68}, Meir and Sharma for $n\le5$ \cite{MS69}, Brown for $n=6$ \cite{Bro91}, Borcea for $n\le7$ \cite{Bor96}, and Brown and Xiang for $n\le8$ \cite{BX99}. For polynomials of sufficiently large degree, Tao \cite{Tao22} proved that Sendov's conjecture holds, and an explicit degree threshold was subsequently given by the author in \cite{Zha26a}.

Regarding the position of the distinguished zero, the case $a=0$ is an immediate consequence of the Gauss--Lucas theorem, whereas Rubinstein  \cite{Rub68} proved the conjecture when $|a|=1$. Further results for zeros on the unit circle were obtained by Goodman, Rahman, and Ratti \cite{GRR69}, and independently by Joyal \cite{Joy69}.

Several other partial results are worth mentioning. Bojanov, Rahman, and Szynal  \cite{BRS85} obtained an asymptotically sharp enlargement of the disk. Miller \cite{Mil93} and V\^aj\^aitu--Zaharescu \cite{VZ93} proved the conjecture when the distinguished zero is sufficiently close to the unit circle. Chijiwa  \cite{Chi11} later made this near-boundary result quantitative, and Kasmalkar \cite{Kas14} subsequently enlarged the range covered by Chijiwa's theorem. At the other end, Bojanov \cite{Boj11} proved the conjecture when $|a|\le 1/(n-1)$. For a zero of fixed modulus $0<|a|<1$, D\'egot  \cite{Deg14} established the conjecture for all sufficiently large degrees, and Chalebgwa  \cite{Cha20} later gave an explicit degree bound.

More recently, Mazur  \cite{Maz26} gave a computer-assisted proof of Sendov's conjecture. Tao \cite{Tao26} then gave a streamlined account of the proof, emphasizing the underlying polynomial identities, discussing the strict form of Sendov's conjecture, and explaining how the argument also yields the strengthening proposed by Phelps and Rodriguez \cite{PR72}.

Motivated by Sendov's conjecture, Tang and the author
\cite[Conjecture~1.10]{TZ25}
(see also \cite[Conjecture~19]{Tao26})
proposed the following family of inequalities strengthening
Sendov's conjecture, with Sendov's conjecture itself appearing
as the weakest limiting case as $\lambda\to\infty$.

\begin{conj}\label{conj:tz}
Let $p$ be a polynomial of degree $n\ge2$ with all its zeros
in the closed unit disk, and let
$\zeta_1,\ldots,\zeta_{n-1}$ be its critical points,
counted with multiplicity.
For every zero $a$ of $p$ and every $\lambda\ge1$,
\begin{equation*}
 \sum_{j=1}^{n-1}\frac{1}{|a-\zeta_j|^\lambda}\ge n-1,
\end{equation*}
where a term with zero denominator in a
reciprocal-distance sum is interpreted as $+\infty$.
\end{conj}

The endpoint $\lambda=1$ is the strongest assertion in Conjecture~\ref{conj:tz}, by the monotonicity of power means. We establish the quadratic case, including all equality cases.

\begin{thm}\label{thm:main}
Let $p$ be a polynomial of degree $n\ge2$ with all its zeros  in the closed unit disk, and let $\zeta_1,\ldots,\zeta_{n-1}$ be its critical points, counted with multiplicity. At every zero $a$ of $p$,
\begin{equation}\label{eq:main}
 \sum_{j=1}^{n-1}\frac{1}{|a-\zeta_j|^2}\ge n-1.
\end{equation}
Equality holds if and only if $p(z)=c(z^n-\omega)$ for some $c\in\mathbb C\setminus\{0\}$ and $\omega\in\mathbb C$ with $|\omega|=1$.
\end{thm}

By the monotonicity of power means, we have
\begin{cor}\label{cor:powers}
Conjecture~\ref{conj:tz} holds for every $\lambda\ge2$. For each such $\lambda$, equality holds if and only if $p(z)=c(z^n-\omega)$ with $c\ne0$ and $|\omega|=1$.
\end{cor}

A Lean~4 formalization of Theorem~\ref{thm:main}, including its equality classification, and Corollary~\ref{cor:powers} is available in~\cite{Zha26b}.

\medskip
\noindent\textbf{Sketch of the proof.}
We briefly describe the proof of Theorem~\ref{thm:main}.
After a rotation, we may assume that the prescribed zero is
$a\in[0,1]$.  For a simple zero $a$, we introduce Tao's reciprocal
coordinates \cite{Tao26}
\[
q_j=\frac{1}{a-\zeta_j},
\qquad
s=\frac{1}{n-1}\sum_{j=1}^{n-1}|q_j|^2.
\]
Thus the desired inequality is equivalent to $s\ge1$, and for
$0<a<1$ we argue by contradiction under the assumption $s\le1$.

The proof combines two communication identities.  The polar identity
yields a second-moment inequality and, in particular, quantitative
control of the endpoint parameter $\beta(1)$.  The origin identity is
then refined by interpolating between the configuration
$q_j=x+iy$ and the actual configuration.  This centered interpolation
retains the variance
\[
v=s-|x+iy|^2
\]
and leads to two explicit scalar inequalities involving only
$n$, $a$, and a small number of auxiliary parameters.

These scalar inequalities are handled in three regimes.
For sufficiently large degree, elementary analytic estimates give a
contradiction directly.  For $6\le n\le10^6$, the remaining parameter
range is covered by finitely many rectangles.  On each rectangle,
monotonicity and convexity reduce the required estimates to exact
rational inequalities with directed rounding.  A finite certificate
of $6593$ rectangles covers the entire range, and two separately
implemented checkers verify the certificate.  Finally, the degrees
$2\le n\le5$ and the endpoint cases $a=0,1$ are treated directly.
The endpoint analysis also determines the equality case
$p(z)=c(z^n-\omega)$, with $c\ne0$ and $|\omega|=1$.

\medskip
\noindent\textbf{Organization of the paper.}
In Section~\ref{sec:notation}, we carry out the normalization and
introduce Tao's reciprocal coordinates together with the second-moment
notation used throughout the proof.  In Section~\ref{sec:identities},
we establish the communication identities.  In Section~\ref{sec:polar},
we develop the polar channel and derive the bounds for $\beta(1)$ and
the endpoint gain.  In Section~\ref{sec:centered}, we introduce the
centered interpolation argument and reduce the problem to two scalar
inequalities.  In Section~\ref{sec:large}, we exclude all sufficiently
large degrees by explicit analytic estimates.  In Section~\ref{sec:finite},
we treat the remaining degrees $6\le n\le10^6$ by a finite exact
verification.  In Section~\ref{sec:completion}, we handle the small
degrees and endpoint cases and complete the proof of the equality
statement.  Finally, Appendix~\ref{app:audit} records the numerical
comparisons, the certificate format, and the details needed to
reproduce the computation.

\medskip
\noindent\textbf{Acknowledgments and AI tools disclosure. } The author is grateful to Terence Tao for very helpful discussions in the comments on his \href{https://terrytao.wordpress.com/2026/08/12/a-digestion-of-the-proof-of-sendovs-conjecture/}{blog post}.

Teng Zhang is supported by the China Scholarship Council, the Young Elite Scientists Sponsorship Program for PhD Students of the China Association for Science and Technology, and the Fundamental Research Funds for the Central Universities at Xi'an Jiaotong University (Grant No.~xzy022024045).

ChatGPT was used for English-language editing, proofreading, grammatical corrections, the preparation of illustrative figures and as an exploratory tool for discussing possible approaches to selected parts of this paper, including, for example, the treatment of the degrees $6\le n\le10^6$. The mathematical arguments and proofs in the original manuscript were independently developed, checked, and written by the author.

The Lean formalization in~\cite{Zha26b} was generated using OpenAI's Codex.
\section{Reduction and notation}\label{sec:notation}

We first normalize the problem and introduce the notation that will be used throughout the proof. Multiplication of $p$ by a nonzero constant does not change its zeros or critical points, while a rotation preserves all relevant distances. We may therefore assume that the prescribed zero is a real number $a\in[0,1].$

Write $p(z)=(z-a)\prod_{j=1}^{n-1}(z-z_j)$, where $|z_j|\le1.$
If $a$ is also a critical point of $p$, then one of the terms in $\sum_{j=1}^{n-1}|a-\zeta_j|^{-2}$
is infinite, and there is nothing to prove. We may therefore assume that $a$ is a simple zero. Following Tao \cite{Tao26}, we parametrize the critical points by
$$
\zeta_j=a-\frac1{q_j},
\qquad q_j\ne0,
\qquad 1\le j\le n-1.
$$
Thus
$$
q_j=\frac1{a-\zeta_j}, \qquad 1\le j\le n-1,
$$
and hence
$$
\frac1{n-1}\sum_{j=1}^{n-1}|a-\zeta_j|^{-2}
=
\frac1{n-1}\sum_{j=1}^{n-1}|q_j|^2.
$$
Consequently, Theorem~\ref{thm:main} is equivalent to proving that
$$
s:=\frac1{n-1}\sum_{j=1}^{n-1}|q_j|^2\ge1.
$$

For $0<a<1$, we shall prove the strict inequality $s>1$. Accordingly, throughout the main part of the argument we work under the counterexample assumption
$
s\le1.
$
The endpoint cases $a=0$ and $a=1$ will be treated separately in Section~\ref{sec:completion}. Thus, until that section, we assume
$
0<a<1.
$

\medskip
\noindent\textbf{Notation.}
We shall use the following notation throughout the argument. For the quantities already appearing in Tao's exposition \cite{Tao26}, we keep the same notation, namely $a$, $z_j$, $\zeta_j$, $q_j$, $\alpha$, $x$, $y$, $\beta(t)$, $F(t)$, and $J$. The quantities $s$, $w_j$, $v$, and $\beta_s(t)$ are introduced here to encode the additional second-moment information.

\begin{itemize}[leftmargin=2.8em,itemsep=5pt,topsep=5pt]

\item $\alpha:=\frac{n-1}{2}(1-a^2)$.

\item $q_j:=\frac{1}{a-\zeta_j}$, where $q_j\ne0$ for $1\le j\le n-1$.

\item $s:=\frac1{n-1}\sum_{j=1}^{n-1}|q_j|^2$, where $s\le1$.

\item $x+iy:=\frac1{n-1}\sum_{j=1}^{n-1}q_j$, where $x,y\in\mathbb R$. By the Cauchy--Schwarz inequality, $|x+iy|\le\sqrt{s}\le1$.

\item $w_j:=q_j-(x+iy)$ for $1\le j\le n-1$. Then $\sum_{j=1}^{n-1}w_j=0$.

\item $v:=\frac1{n-1}\sum_{j=1}^{n-1}|w_j|^2=s-|x+iy|^2\le1-|x+iy|^2$.

\item $\beta(t):=1-2atx+a^2t^2=1-x^2+(x-at)^2$, where $0\le t\le1$.

\item $\beta_s(t):=1-2atx+a^2st^2
=\frac1{n-1}\sum_{j=1}^{n-1}|1-atq_j|^2$, where $0\le t\le1$. Since $s\le1$, we have $\beta_s(t)\le\beta(t)$.

\item $F(t):=\prod_{j=1}^{n-1}(1-atq_j)$.

\item $J:=\prod_{j=1}^{n-1}z_jq_j$. Since $|z_j|\le1$, the arithmetic--geometric mean inequality gives
$|J|\le\prod_{j=1}^{n-1}|q_j|\le s^{(n-1)/2}\le1$.

\end{itemize}

\section{Communication identities and preliminary estimates}
\label{sec:identities}

\begin{lem}[Communication identities]\label{lem:origin}
\begin{enumerate}[label=\textup{(\roman*)},leftmargin=2.3em,itemsep=0.4em]

\item (First origin identity) One has
\begin{equation}\label{eq:origin}
(-1)^{n-1}J=n\int_0^1F(t)\dd t.
\end{equation}

\item (Polar identity) One has
\begin{equation}\label{eq:polaridentity}
\prod_{j=1}^{n-1}\frac{1-az_j}{a-z_j}
=
\int_0^1\prod_{j=1}^{n-1}
\bigl(a+t(1-a^2)q_j\bigr)\dd t.
\end{equation}

\item (Differential origin identity) If
$R:=a^2\int_0^1t\sum_{j=1}^{n-1}q_j^2 \prod_{k\ne j}(1-atq_k)\dd t$,
then
\begin{equation}\label{eq:derivative}
1=F(1)+(n-1)a(x+iy)\int_0^1F(t)\dd t+R.
\end{equation}

\end{enumerate}
\end{lem}

\begin{proof}
For (i), this is exactly \cite[Lemma~6, Eq.~(9)]{Tao26}.

For (ii), this is the polar identity in
\cite[Lemma~6, Eq.~(8)]{Tao26}.

For (iii), differentiate the product defining $F(t)$. Using the elementary identity
$$
 \prod_{k\ne j}(1-atq_k)
 =F(t)+atq_j\prod_{k\ne j}(1-atq_k),
$$
we obtain
$$
 F'(t)=-(n-1)a(x+iy)F(t)-a^2t\sum_jq_j^2\prod_{k\ne j}(1-atq_k).
$$
Integrating over $[0,1]$ gives \eqref{eq:derivative}.
\end{proof}

We shall repeatedly encounter products with one factor omitted. For $n\ge3$, any nonnegative numbers $b_1,\ldots,b_{n-1}$ satisfy
\begin{equation}\label{eq:deleted}
\prod_{k\ne j}b_k
\le\left(\frac{\sum_kb_k^2}{n-2}\right)^{(n-2)/2}
\le\frac53\left(\frac1{n-1}\sum_kb_k^2\right)^{(n-2)/2}.
\end{equation}
Indeed,
$$
 \left(\frac{n-1}{n-2}\right)^{(n-2)/2}\le\sqrt\e<\frac53.
$$
It follows that the remainder term in Lemma~\ref{lem:origin} admits the following direct estimates.

\begin{lem}[Direct estimates]\label{lem:rawerrors}
Let $n\ge3$, and let $R$ be defined as in Lemma~\ref{lem:origin}. Then
\begin{equation*}
|F(1)|\le \beta_s(1)^{(n-1)/2},\qquad
|R|\le\frac53a^2(n-1)\int_0^1t\beta_s(t)^{(n-2)/2}\dd t.
\end{equation*}
\end{lem}

\begin{proof}
By definition, the average of the quantities $|1-atq_j|^2$ is $\beta_s(t)$. The first estimate follows from the arithmetic--geometric mean inequality applied to the factors of $F(1)$. For the second, apply \eqref{eq:deleted} to each product appearing in $R$ and use $\sum_j|q_j|^2=(n-1)s\le n-1$.
\end{proof}

\section{The polar channel}\label{sec:polar}

We begin with the polar channel. The argument follows the one leading to the raw polar inequality in \cite[Proposition~10(i)]{Tao26}, which is derived from the polar communication identity \cite[Lemma~6, (8)]{Tao26}. The only difference is that, rather than replacing the second moment $\frac1{n-1}\sum_{j=1}^{n-1}|q_j|^2$
by $1$, we keep its exact value $s$. This gives the following second-moment version of Tao's raw polar inequality.
\begin{lem}[Second-moment polar inequality]\label{lem:polar}
One has
\begin{equation}\label{eq:polarraw}
1\le\int_0^1
\left(a^2+2a(1-a^2)xt+(1-a^2)^2st^2\right)^{(n-1)/2}\dd t.
\end{equation}
\end{lem}

\begin{proof}
From the polar identity~\eqref{eq:polaridentity}, we have
\[
 \int_0^1\prod_{j=1}^{n-1}\bigl(a+(1-a^2)tq_j\bigr)\dd t
 =\prod_{j=1}^{n-1}\frac{1-az_j}{a-z_j}.
\]
Since $a$ is a simple zero, $a-z_j\ne0$ for every $j$. Moreover, since $|z_j|\le1$,
\[
 |1-az_j|^2-|a-z_j|^2
 =(1-a^2)(1-|z_j|^2)\ge0,
\]
and hence
\[
 \left|\frac{1-az_j}{a-z_j}\right|\ge1.
\]
It follows from the triangle inequality that
\begin{equation}\label{eq:polar-triangle}
1\le
\int_0^1
\prod_{j=1}^{n-1}
\left|a+(1-a^2)tq_j\right|
\dd t.
\end{equation}

For each fixed $t\in[0,1]$, the arithmetic--geometric mean inequality applied to the nonnegative numbers
$\left|a+(1-a^2)tq_j\right|^2$ gives
\begin{equation}\label{eq:polar-amgm}
\prod_{j=1}^{n-1}
\left|a+(1-a^2)tq_j\right|
\le
\left(
\frac1{n-1}
\sum_{j=1}^{n-1}
\left|a+(1-a^2)tq_j\right|^2
\right)^{(n-1)/2}.
\end{equation}

Using
$x+iy=\frac1{n-1}\sum_{j=1}^{n-1}q_j$
and
$s=\frac1{n-1}\sum_{j=1}^{n-1}|q_j|^2$,
we obtain
\begin{equation}\label{eq:polar-second-moment}
\begin{aligned}
\frac1{n-1}
\sum_{j=1}^{n-1}
\left|a+(1-a^2)tq_j\right|^2
&=
a^2
+2a(1-a^2)t
\Re\left(\frac1{n-1}\sum_{j=1}^{n-1}q_j\right)
+(1-a^2)^2t^2
\frac1{n-1}\sum_{j=1}^{n-1}|q_j|^2\\
&=
a^2+2a(1-a^2)xt+(1-a^2)^2st^2.
\end{aligned}
\end{equation}
Combining \eqref{eq:polar-triangle}, \eqref{eq:polar-amgm}, and
\eqref{eq:polar-second-moment}, we obtain
\[
1\le\int_0^1
 \left(a^2+2a(1-a^2)xt+(1-a^2)^2st^2\right)^{(n-1)/2}\dd t,
\]
which is precisely \eqref{eq:polarraw}.
\end{proof}

We next derive the bounds on Tao's parameter $\beta(1)$ that will be needed below. The argument leading to the logarithmic estimate follows \cite[Proposition~10(ii)--(iii)]{Tao26}. The difference is that our polar inequality retains the exact second moment $s$; nevertheless, the assumption $s\le1$ still leads to the same exponential inequality. For the rational bound, we use a simpler but weaker estimate than Tao's bound $\beta(1)<\alpha/(3+\alpha)$, which is sufficient for our purposes. The strict inequalities below remain valid even when $s=1$.

\begin{lem}[Bounds for $\beta(1)$]\label{lem:beta}
One has
\begin{equation}\label{eq:betabound}
 0<\beta(1)<\frac{\alpha}{1+\alpha}<1,
 \qquad x>\frac a2.
\end{equation}
For $\alpha\ge1$ one also has
\begin{equation}\label{eq:betalog}
 1-\beta(1)>\frac{\log\alpha}{\alpha}.
\end{equation}
\end{lem}

\begin{proof}
Since $s\le1$, the quadratic inside the integral in \eqref{eq:polarraw} is bounded above by the same expression with $s$ replaced by $1$. For $0<t<1$, the inequality $t^2<t$, together with $2-\beta(1)=2ax+1-a^2$, makes this bound strict. Using $1+u\le\e^u$ and $\alpha=\frac{n-1}{2}(1-a^2)$, we therefore obtain
\begin{equation}\label{eq:exppolar}
 1<\int_0^1\exp\{\alpha[-1+(2-\beta(1))t]\}\dd t.
\end{equation}
Notice that the strict inequality persists when $s=1$, precisely because $t^2<t$ for $0<t<1$.

If $\beta(1)\ge1$, then the exponent in \eqref{eq:exppolar} is nonpositive on $[0,1]$ and is strictly negative on a set of positive measure, contradicting \eqref{eq:exppolar}. Hence $\beta(1)<1$, or equivalently $x>a/2$. On the other hand,
\[
 \beta(1)=1-x^2+(x-a)^2>0,
\]
because $|x|\le1$ and $a<1$.

Evaluating the integral in \eqref{eq:exppolar} gives
\begin{equation}\label{eq:exppolar-evaluated}
 \e^{\alpha(1-\beta(1))}-\e^{-\alpha}
 >\alpha(2-\beta(1)).
\end{equation}
Since $\beta(1)<1$, it follows in particular that $\e^{\alpha(1-\beta(1))}>\alpha$. For $\alpha\ge1$, taking logarithms yields \eqref{eq:betalog}.

It remains to prove the rational bound in \eqref{eq:betabound}. Put $z=\alpha/(1+\alpha)$. Since $z<\log(1+\alpha)$, we have $\e^z<1+\alpha$, and therefore
\[
 \e^z-1-z<\frac{z^2\e^z}{2}
       <\frac{\alpha^2}{2(1+\alpha)}.
\]
Also, for $0<t\le\alpha$ one has $1-\e^{-t}>t/(1+t)\ge t/(1+\alpha)$, and hence
\[
 \alpha-1+\e^{-\alpha}
 =\int_0^\alpha(1-\e^{-t})\dd t
 >\int_0^\alpha\frac{t}{1+\alpha}\dd t
 =\frac{\alpha^2}{2(1+\alpha)}.
\]
Suppose, to the contrary, that $\alpha(1-\beta(1))\le z$. Since $\beta(1)<1$, we have $0<\alpha(1-\beta(1))\le z$, and the function $\e^w-1-w$ is increasing for $w\ge0$. Combining the preceding two estimates therefore gives
\[
 \e^{\alpha(1-\beta(1))}-\e^{-\alpha}
 <\alpha+\alpha(1-\beta(1))
 =\alpha(2-\beta(1)),
\]
contradicting \eqref{eq:exppolar-evaluated}. Thus $\alpha(1-\beta(1))>z$, and hence $1-\beta(1)>1/(1+\alpha)$. This is equivalent to
$\beta(1)<\alpha/(1+\alpha)$ and completes the proof.
\end{proof}

The raw polar inequality contains more information than the exponential form used above. In particular, by retaining the exact second moment $s$, we can keep track of the value of the quadratic integrand at $t=1$. To this end, define
\[
 \rho:=(1-a^2)s-1+2ax
      =1-\beta(1)-(1-a^2)(1-s).
\]
Thus $\rho$ measures the endpoint gain that remains after accounting for the possible deficit $1-s$ in the second moment. This quantity does not appear in Tao's argument, where the second moment is replaced by its upper bound $1$ before the polar inequality is simplified.

\begin{lem}[Endpoint gain]\label{lem:gap}
One has $0<\rho\le1-\beta(1)$ and
\begin{equation}\label{eq:chordpolar}
 \left(1+(1-a^2)\rho\right)^{(n+1)/2}
 -a^{n+1}
 -\frac{n+1}{2}(1-a^2)(1+\rho)\ge0.
\end{equation}
If $
 \frac{n+1}{2}(1-a^2)>1,$
then
\begin{equation}\label{eq:rhobounds}
 \rho\ge
 \frac{\log\!\left(\frac{n+1}{2}(1-a^2)\right)}{\frac{n+1}{2}(1-a^2)},
 \qquad \beta_s(1)\le1-\rho,
 \qquad x^2\ge\rho.
\end{equation}
\end{lem}

\begin{proof}
Set
\[
 H(t)=a^2+2a(1-a^2)xt+(1-a^2)^2st^2.
\]
This is a convex quadratic on $[0,1]$, with
\[
 H(0)=a^2,
 \qquad
 H(1)=1+(1-a^2)\rho.
\]
Hence $H$ lies below the chord joining its endpoint values, namely
\[
 \ell_\rho(t)=a^2+(1-a^2)(1+\rho)t.
\]
Figure~\ref{fig:polar-chord} summarizes the geometric step used here and in the contradiction argument below.

\begin{figure}[H]
\centering
\begin{tikzpicture}[x=1cm,y=1cm,line cap=round,line join=round,>=Latex]
  \begin{scope}[shift={(0,0)}]
    \node[font=\small\bfseries, anchor=west] at (-.10,3.55) {(a) Convex $H$ and its chord};
    \draw[->, black!55, line width=.65pt] (0,0)--(4.65,0) node[below, black!75] {$t$};
    \draw[->, black!55, line width=.65pt] (0,0)--(0,3.25);
    \draw[black!25, dashed] (4.15,0)--(4.15,2.78);
    \draw[black!25, dashed] (0,0.72)--(4.15,0.72);

    \coordinate (A) at (0,0.72);
    \coordinate (B) at (4.15,2.78);
    \draw[RoyalBlue!85!black, line width=1.15pt] (A)--(B);
    \draw[ForestGreen!70!black, line width=1.2pt]
      plot[smooth] coordinates {(0,.72) (.70,.86) (1.45,1.08) (2.20,1.40) (3.00,1.85) (3.55,2.27) (4.15,2.78)};
    \fill[black!70] (A) circle (1.7pt);
    \fill[black!70] (B) circle (1.7pt);
    \node[font=\scriptsize, anchor=east] at (-.06,.72) {$a^2$};
    \node[font=\scriptsize, anchor=west, align=left] at (4.22,2.78) {$1+(1-a^2)\rho$};
    \node[font=\scriptsize, RoyalBlue!80!black, rotate=26] at (2.78,2.22) {$\ell_\rho(t)$};
    \node[font=\scriptsize, ForestGreen!55!black] at (2.47,1.18) {$H(t)$};
    \node[font=\scriptsize, black!65] at (2.16,.42) {$H(t)\le \ell_\rho(t)$};
    \node[font=\scriptsize, black!65] at (4.15,-.24) {$1$};
  \end{scope}

  \begin{scope}[shift={(6.55,0)}]
    \node[font=\small\bfseries, anchor=west] at (-.10,3.55) {(b) Hypothetical $\rho\le0$};
    \draw[->, black!55, line width=.65pt] (0,0)--(4.65,0) node[below, black!75] {$t$};
    \draw[->, black!55, line width=.65pt] (0,0)--(0,3.25);
    \draw[BrickRed!70!black, dashed, line width=.8pt] (0,2.55)--(4.42,2.55)
      node[right, font=\scriptsize, BrickRed!70!black] {$1$};
    \draw[black!25, dashed] (4.15,0)--(4.15,2.40);

    \coordinate (C) at (0,.72);
    \coordinate (D) at (4.15,2.34);
    \draw[RoyalBlue!85!black, line width=1.15pt] (C)--(D);
    \draw[ForestGreen!70!black, line width=1.2pt]
      plot[smooth] coordinates {(0,.72) (.75,.79) (1.48,.92) (2.20,1.15) (2.95,1.48) (3.55,1.86) (4.15,2.34)};
    \fill[black!70] (C) circle (1.7pt);
    \fill[black!70] (D) circle (1.7pt);
    \node[font=\scriptsize, anchor=east] at (-.06,.72) {$a^2$};
    \node[font=\scriptsize, anchor=west] at (4.22,2.34) {$\ell_\rho(1)\le1$};
    \node[font=\scriptsize, RoyalBlue!80!black, rotate=20] at (2.75,1.93) {$\ell_\rho(t)$};
    \node[font=\scriptsize, ForestGreen!55!black] at (2.45,1.02) {$H(t)$};
    \node[font=\scriptsize, black!65, align=center] at (2.15,.40) {$H(t)\le\ell_\rho(t)\le1$};
    \node[font=\scriptsize, black!65] at (4.15,-.24) {$1$};
  \end{scope}
\end{tikzpicture}
\caption{The chord comparison in Lemma~\ref{lem:gap}.  Convexity places $H$ below the affine chord $\ell_\rho$.  If $\rho\le0$, the whole chord is at or below the level $1$, with strict inequality on a set of positive measure; this is incompatible with the raw polar inequality.}
\label{fig:polar-chord}
\end{figure}
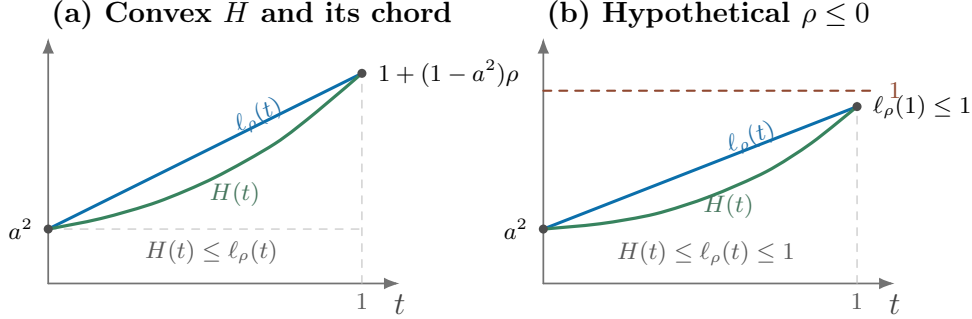

We first show that $\rho>0$. If $\rho\le0$, then $\ell_\rho(t)\le1$ for $0\le t\le1$, and the inequality is strict on a set of positive measure since $a<1$. As $H(t)\le\ell_\rho(t)$, this would give
\[
 \int_0^1 H(t)^{(n-1)/2}\dd t<1,
\]
contradicting the raw polar inequality~\eqref{eq:polarraw}. Thus $\rho>0$. From the definition of $\rho$ and the assumption $s\le1$, we also have
\[
 \rho
 =1-\beta(1)-(1-a^2)(1-s)
 \le1-\beta(1).
\]

Since $H(t)\le\ell_\rho(t)$, the raw polar inequality~\eqref{eq:polarraw}  gives
\begin{equation}\label{eq:chord-integral}
 1\le\int_0^1\ell_\rho(t)^{(n-1)/2}\dd t.
\end{equation}
The integral can be evaluated explicitly:
\[
 \int_0^1\ell_\rho(t)^{(n-1)/2}\dd t
 =
 \frac{2}{(n+1)(1-a^2)(1+\rho)}
 \left[
 \left(1+(1-a^2)\rho\right)^{(n+1)/2}
 -a^{n+1}
 \right].
\]
Substituting this into \eqref{eq:chord-integral} gives
\eqref{eq:chordpolar}.

In particular, \eqref{eq:chordpolar} implies
\[
 \left(1+(1-a^2)\rho\right)^{(n+1)/2}
 \ge \frac{n+1}{2}(1-a^2).
\]
Suppose now that $
 \frac{n+1}{2}(1-a^2)>1.$
Taking logarithms and using $\log(1+t)\le t$ gives
\[
 \log\!\left(\frac{n+1}{2}(1-a^2)\right)
 \le
 \frac{n+1}{2}\log\!\left(1+(1-a^2)\rho\right)
 \le
 \frac{n+1}{2}(1-a^2)\rho,
\]
which proves the first estimate in \eqref{eq:rhobounds}. Finally,
\[
 \beta_s(1)=s-\rho\le1-\rho,
\]
while
\[
 \rho\le1-\beta(1)=2ax-a^2=x^2-(x-a)^2\le x^2.
\]
This proves the remaining two estimates.
\end{proof}

\section{The centered origin channel}\label{sec:centered}

The direct estimates in Lemma~\ref{lem:rawerrors} do not make use of the variance information. We now refine them by exploiting the fact that the $q_j$ are concentrated around their mean.

Instead of expanding in elementary symmetric functions, we interpolate between the constant configuration $q_j=x+iy$ and the actual configuration.  The deformation is visualized in Figure~\ref{fig:centered-interpolation}: every point moves radially away from the common mean, while the mean itself stays fixed and the variance scales by $\theta^2$.

\begin{figure}[H]
\centering
\begin{tikzpicture}[x=1.22cm,y=1.22cm,line cap=round,line join=round,>=Latex]
  \coordinate (M) at (0.35,0.18);

  \draw[RoyalBlue!20, line width=.85pt, dashed, rounded corners=9pt]
    (-1.32,1.95) -- (2.53,1.53) -- (1.43,-1.61) -- (-1.24,-1.19) -- cycle;
  \draw[ForestGreen!32, line width=.8pt, dashed, rounded corners=8pt]
    (-.57,1.16) -- (1.55,.92) -- (.94,-.80) -- (-.52,-.58) -- cycle;

  \foreach \x/\y in {2.35/1.38,-1.15/1.78,-1.05/-1.02,1.25/-1.42}{
    \draw[black!25, dashed, line width=.65pt] (M)--(\x,\y);
  }

  \fill[RoyalBlue] (2.35,1.38) circle (2.7pt);
  \fill[RoyalBlue] (-1.15,1.78) circle (2.7pt);
  \fill[RoyalBlue] (-1.05,-1.02) circle (2.7pt);
  \fill[RoyalBlue] (1.25,-1.42) circle (2.7pt);
  \node[RoyalBlue!80!black, font=\scriptsize, anchor=south west] at (2.37,1.40) {$q_1$};
  \node[RoyalBlue!80!black, font=\scriptsize, anchor=south east] at (-1.18,1.80) {$q_2$};
  \node[RoyalBlue!80!black, font=\scriptsize, anchor=north east] at (-1.08,-1.05) {$q_3$};
  \node[RoyalBlue!80!black, font=\scriptsize, anchor=north west] at (1.28,-1.45) {$q_4$};

  \fill[ForestGreen!70!black] (1.45,.84) circle (2.35pt);
  \fill[ForestGreen!70!black] (-.475,1.06) circle (2.35pt);
  \fill[ForestGreen!70!black] (-.42,-.48) circle (2.35pt);
  \fill[ForestGreen!70!black] (.845,-.70) circle (2.35pt);
  \node[ForestGreen!55!black, font=\scriptsize, anchor=south west] at (1.48,.84) {$q_1(\theta)$};

  \draw[->, RoyalBlue!65!black, line width=.72pt] (.82,.46)--(1.16,.64);
  \draw[->, RoyalBlue!65!black, line width=.72pt] (-.01,.60)--(-.25,.84);
  \draw[->, RoyalBlue!65!black, line width=.72pt] (-.01,-.12)--(-.25,-.32);
  \draw[->, RoyalBlue!65!black, line width=.72pt] (.58,-.22)--(.73,-.44);

  \fill[BrickRed] (M) circle (3.2pt);
  \node[BrickRed!85!black, font=\small, anchor=south west] at ($(M)+(.09,.07)$) {$m=x+iy$};
  \node[font=\scriptsize, black!58, anchor=north] at (.48,-1.90) {complex $q$-plane};

  \begin{scope}[shift={(3.08,.40)}]
    \draw[black!16, rounded corners=4pt, fill=black!1] (-.05,-1.48) rectangle (4.12,1.45);
    \node[anchor=west, font=\small] at (.16,1.04) {$q_j(\theta)=m+\theta w_j$};
    \node[anchor=west, font=\small] at (.16,.43)
      {$\displaystyle \frac1{n-1}\sum_j q_j(\theta)=m$};
    \node[anchor=west, font=\small] at (.16,-.25)
      {$\displaystyle \frac1{n-1}\sum_j|q_j(\theta)-m|^2=\theta^2v$};
    \draw[RoyalBlue, line width=1.05pt] (.18,-.91)--(.72,-.91);
    \node[anchor=west, font=\scriptsize] at (.84,-.91) {actual configuration $(\theta=1)$};
    \fill[ForestGreen!70!black] (.30,-1.24) circle (2.1pt);
    \node[anchor=west, font=\scriptsize] at (.42,-1.24) {intermediate configuration};
  \end{scope}
\end{tikzpicture}
\caption{Centered interpolation in reciprocal coordinates.  Writing $m=x+iy$ and $w_j=q_j-m$, the path $q_j(\theta)=m+\theta w_j$ contracts the configuration to its common mean at $\theta=0$ and recovers the original configuration at $\theta=1$.  The mean is fixed, while the variance is exactly $\theta^2v$.}
\label{fig:centered-interpolation}
\end{figure}
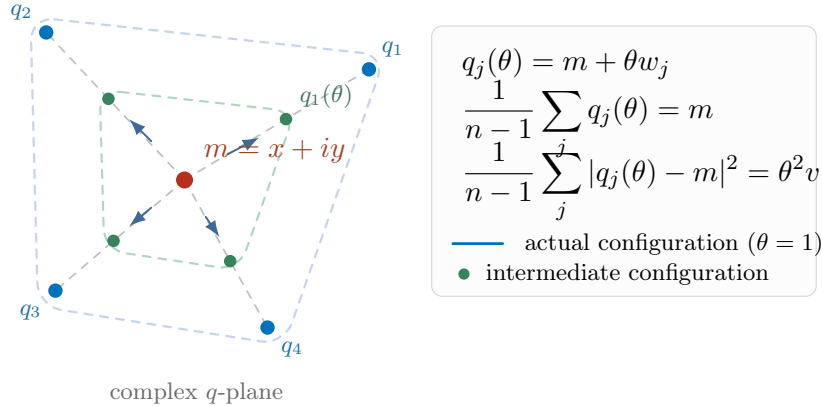

\begin{lem}[Centered interpolation estimate]\label{lem:centered}
For $n\ge3$, put $D(t)=1-at(x+iy)$. Then
\begin{equation}\label{eq:centered}
 |F(t)-D(t)^{n-1}|
 \le\frac{5(n-1)a^2t^2v}{6|D(t)|}\,\beta_s(t)^{(n-2)/2}
 \qquad(0\le t\le1).
\end{equation}
\end{lem}

\begin{proof}
Set
\[
 G_\theta(t)=\prod_{j=1}^{n-1}\bigl(D(t)-at\theta w_j\bigr),
 \qquad 0\le\theta\le1.
\]
Thus $G_0(t)=D(t)^{n-1}$ and $G_1(t)=F(t)$. The cancellation $\sum_jw_j=0$ gives the identity
\begin{equation}\label{eq:interpolation}
 D(t)\,\partial_\theta G_\theta(t)
 =-a^2t^2\theta\sum_jw_j^2
          \prod_{k\ne j}\bigl(D(t)-at\theta w_k\bigr).
\end{equation}
Indeed, after differentiating $G_\theta(t)$, replace $D(t)$ in the $j$th summand by
$D(t)-at\theta w_j+at\theta w_j$. The terms containing $G_\theta(t)$ then cancel because $\sum_jw_j=0$.

Moreover,
\[
 \frac1{n-1}\sum_j|(x+iy)+\theta w_j|^2
 =|x+iy|^2+\theta^2v\le s,
 \qquad
 \frac1{n-1}\sum_j((x+iy)+\theta w_j)=x+iy.
\]
It follows that the mean square of the factors $D(t)-at\theta w_j$ is at most $\beta_s(t)$. Applying \eqref{eq:deleted} to each product in \eqref{eq:interpolation}, and using $\sum_j|w_j^2|=(n-1)v$, we obtain
\[
 \left|\partial_\theta G_\theta(t)\right|
 \le
 \frac{5(n-1)a^2t^2\theta v}{3|D(t)|}\,
 \beta_s(t)^{(n-2)/2}.
\]
Finally, $|D(t)|>0$ because $at|x+iy|<1$. Integrating the preceding estimate over $0\le\theta\le1$, and using $G_0(t)=D(t)^{n-1}$ and $G_1(t)=F(t)$, gives \eqref{eq:centered}.
\end{proof}

Define the positive coefficient
\begin{equation}\label{eq:Cdef}
 C_n(a,x):=\frac{5a^3n(n-1)}6
      \int_0^1\frac{t^2\beta(t)^{(n-2)/2}}{1-axt}\dd t.
\end{equation}
The denominator is positive since $ax<1$.

\begin{prop}[Scalar inequalities]\label{prop:scalar}
Let $n\ge3$. Then
\begin{equation}\label{eq:centerscalar}
 1\le a|x+iy|
 +C_n(a,x)|x+iy|\bigl(1-|x+iy|^2\bigr)
 +\beta(1)^{n/2}
\end{equation}
and
\begin{equation}\label{eq:directscalar}
 1\le\beta(1)^{(n-1)/2}+\frac{(n-1)a}{n}
       +\frac53a^2(n-1)\int_0^1t\beta(t)^{(n-2)/2}\dd t.
\end{equation}
\end{prop}

\begin{proof}
By \eqref{eq:betabound}, $x+iy\ne0$. Since $D(t)=1-at(x+iy)$, direct integration gives
\begin{equation}\label{eq:Dintegral}
 n\int_0^1D(t)^{n-1}\dd t
 =\frac{1-D(1)^n}{a(x+iy)}.
\end{equation}
Integrating \eqref{eq:centered} over $[0,1]$ and multiplying by $a|x+iy|n$, we obtain
\[
 a|x+iy|n
 \left|
 \int_0^1F(t)\dd t-\int_0^1D(t)^{n-1}\dd t
 \right|
 \le
 \frac{5a^3n(n-1)|x+iy|v}{6}
 \int_0^1
 \frac{t^2\beta_s(t)^{(n-2)/2}}{|D(t)|}\dd t.
\]
Using $\beta_s(t)\le\beta(t)$ and
$|D(t)|\ge1-axt$, and then applying \eqref{eq:Cdef} and \eqref{eq:Dintegral}, this gives
\begin{equation}\label{eq:centered-integrated}
 \left|a(x+iy)n\int_0^1F(t)\dd t
       -\bigl(1-D(1)^n\bigr)\right|
 \le C_n(a,x)|x+iy|v.
\end{equation}
Moreover,
\[
 |D(1)|^2
 =|1-a(x+iy)|^2
 \le\beta(1).
\]
By the first origin identity~\eqref{eq:origin},
\[
 \left|n\int_0^1F(t)\dd t\right|
 =|J|\le1.
\]
Combining this with \eqref{eq:centered-integrated}, and using
$v\le1-|x+iy|^2$ and $|D(1)|^n\le\beta(1)^{n/2}$, yields
\eqref{eq:centerscalar}.

For the second estimate, the differential origin identity~\eqref{eq:derivative} gives
\[
 1\le |F(1)|
 +(n-1)a|x+iy|
 \left|\int_0^1F(t)\dd t\right|
 +|R|.
\]
By Lemma~\ref{lem:rawerrors}, $\beta_s\le\beta$, and $|x+iy|\le1$, we have
\[
 |F(1)|\le\beta(1)^{(n-1)/2},
 \qquad
 |R|\le\frac53a^2(n-1)\int_0^1t\beta(t)^{(n-2)/2}\dd t.
\]
Finally, \eqref{eq:origin} and $|J|\le1$ give
\[
 \left|\int_0^1F(t)\dd t\right|\le\frac1n.
\]
Substituting these three estimates into the preceding inequality gives
\eqref{eq:directscalar}.
\end{proof}

It remains only to carry out a one-dimensional optimization. Write temporarily $r=|x+iy|\in[0,1]$. For $a,C\ge0$,
\begin{equation}\label{eq:cubicmax}
 \Psi(a,C):=\max_{0\le r\le1}\{ar+Cr(1-r^2)\}
 =\begin{cases}
 a,&2C\le a,\\[1mm]
 \displaystyle\frac{2(a+C)^{3/2}}{3\sqrt{3C}},&2C>a.
 \end{cases}
\end{equation}
This is obtained by maximizing the cubic $(a+C)r-Cr^3$ on $[0,1]$. The function $\Psi$ is nondecreasing in each variable. Consequently, any upper bound strictly less than one for either of the right-hand sides in Proposition~\ref{prop:scalar} rules out a counterexample.

\section{Analytic elimination for large degree}\label{sec:large}

We first treat the range of sufficiently large degrees. The threshold is chosen for convenience, with ample room in the elementary estimates below. Throughout this section we assume
\[
 n-1\ge10^6.
\]
Appendix~\ref{app:constants} records the exact rational checks at the endpoint, together with the monotonicity arguments used to extend them to the full range.

\subsection{Zeros near the unit circle}
In this range, the centered estimate alone is sufficient.

\begin{prop}\label{prop:near}
The inequality $s\le1$ is impossible when
$n\ge1000001$ and $0<1-a^2\le1/10$.
\end{prop}

\begin{proof}
We have $a^2\ge9/10$, $a\ge9/10$, and $\alpha\le(n-1)/20$. Put
\[
 \nu:=\frac{n-4}{2}.
\]
We first establish the two tail estimates
\begin{equation}\label{eq:tailpowers}
 \beta(1)^\nu\le(n-1)^{-4},\qquad
 (133/160)^\nu\le(n-1)^{-4}.
\end{equation}
Suppose first that $\alpha\le\sqrt{n-1}$. By \eqref{eq:betabound} and $\nu\ge(n-1)/3$,
\[
 \nu(1-\beta(1))
 >\frac{(n-1)/3}{1+\sqrt{n-1}}
 \ge\frac{\sqrt{n-1}}6>4\log(n-1).
\]
It follows that $\beta(1)^\nu<(n-1)^{-4}$. If instead $\sqrt{n-1}\le\alpha\le(n-1)/20$, then $\nu/\alpha\ge9$, and \eqref{eq:betalog} gives
\[
 \beta(1)^\nu
 \le\exp\{-\nu(1-\beta(1))\}
 <\alpha^{-9}\le(n-1)^{-9/2}.
\]
This proves the first estimate in \eqref{eq:tailpowers}. For the second, we use $133/160=1-27/160$ and obtain
\[
 (133/160)^\nu\le\exp(-27\nu/160)
 \le\exp(-9(n-1)/160)\le(n-1)^{-4}.
\]

We now estimate the coefficient $C_n(a,x)$ in \eqref{eq:Cdef} by splitting the integral at $t=1/4$. For $0\le t\le1/4$, the inequality $x>a/2$ gives
\[
 \beta(t)\le1-\frac{3a^2t}{4}\le\exp(-3a^2t/4),
 \qquad 1-axt\ge\frac34.
\]
Using $\int_0^\infty t^2\e^{-kt}\dd t=2/k^3$, the contribution of this interval to $C_n(a,x)$ is at most
\[
 \frac{2048(5/3)n(n-1)}{81(n-2)^3a^3}
 <\frac{60}{n-1}.
\]
On $[1/4,1]$, the convexity of $\beta$ gives
\[
 \beta(t)\le\max\{\beta(1/4),\beta(1)\}
       \le\max\{133/160,\beta(1)\}.
\]
Furthermore, $\beta(t)\le2(1-axt)$. Hence, by \eqref{eq:tailpowers}, the contribution from this interval is at most
\[
 \frac53a^3n(n-1)
 \bigl\{(133/160)^\nu+\beta(1)^\nu\bigr\}
 <\frac7{(n-1)^2}.
\]
Combining the two estimates, we obtain
\begin{equation}\label{eq:Cnear}
 C_n(a,x)<\frac{61}{n-1}<\frac a2.
\end{equation}

It remains to control the endpoint term in \eqref{eq:centerscalar}. Set $E:=\beta(1)^{n/2}$. If $\alpha\ge1$, then \eqref{eq:tailpowers} and $1-a^2=2\alpha/(n-1)$ imply
\[
 \frac{E}{1-a^2}\le\frac1{2(n-1)^3}<\frac14.
\]
If $0<\alpha\le1$, then \eqref{eq:betabound} gives $\beta(1)\le\min\{\alpha,1/2\}$, and hence
\[
 \frac{E}{1-a^2}\le(n-1)2^{-n/2}<\frac14.
\]
Thus, in both cases, $E<(1-a^2)/4$. Together with \eqref{eq:Cnear}, the formula \eqref{eq:cubicmax} shows that the right-hand side of \eqref{eq:centerscalar} is less than
\[
 a+\frac{1-a^2}{4}<1,
\]
which contradicts \eqref{eq:centerscalar}.
\end{proof}

\subsection{The complementary range}
We now turn to the complementary range, where the differential origin identity is sufficient. The endpoint gain obtained in Lemma~\ref{lem:gap} controls the increasing part of the remainder integral.

\begin{prop}\label{prop:away}
The inequality $s\le1$ is impossible when $n\ge1{,}000{,}001$ and $1-a^2\ge1/10$.
\end{prop}

\begin{proof}
For this proof only, put
\[
 \Lambda:=\frac{n+1}{2}(1-a^2),\qquad
 \sigma:=\frac{n-2}{n+1}.
\]
Then $\Lambda>1$. By Lemmas~\ref{lem:rawerrors} and \ref{lem:gap},
\begin{equation}\label{eq:Flarge}
 |F(1)|\le\exp(-(n-1)\rho/2)
 \le\Lambda^{-(n-1)/((n+1)(1-a^2))}.
\end{equation}

Let $t_0=x/(as)$. Since every $q_j$ is nonzero, we have $s>0$. The quadratic $\beta_s(t)$ is decreasing for $t\le t_0$ and increasing for $t\ge t_0$. We therefore split the integral defining $R$ at $\min\{1,t_0\}$, and denote the two parts by $R_{\rm dec}$ and $R_{\rm inc}$; when $t_0\ge1$, the latter is absent.

On the decreasing part, $\beta_s(t)\le1-axt\le\e^{-axt}$. Hence
\begin{equation}\label{eq:Rdec}
 |R_{\rm dec}|
 \le\frac{20(n-1)}{3(n-2)^2x^2}
 <\frac7{(n-1)x^2}.
\end{equation}
On the increasing part, $\beta_s(t)\le\beta_s(1)\le1-\rho$, and therefore
\begin{equation}\label{eq:Rinc}
 |R_{\rm inc}|\le\frac53a^2(n-1)\Lambda^{-\sigma/(1-a^2)}.
\end{equation}
We shall show that $|F(1)|+|R|<1-a$.

Suppose first that $0<a\le1/2$. Then $1-a^2\ge3/4$ and $\Lambda\ge3(n-1)/8$. Since $x^2\ge\rho\ge\log\Lambda/\Lambda$, \eqref{eq:Rdec} gives
\[
 |R_{\rm dec}|<\frac{7\Lambda}{(n-1)\log\Lambda}<0.292.
\]
Write $\varepsilon=a^2$ and $\eta=3/(n+1)$, so that $\sigma=1-\eta$. Then
\[
 \left(\frac{2}{1-a^2}\right)^{\sigma/(1-a^2)}
 \le(8/3)^{4/3}<4,
 \qquad
 (n-1)^{\eta/(1-a^2)}
 \le(n-1)^{4/(n+1)}<1.001.
\]
Using $\Lambda\ge(n-1)(1-a^2)/2$ in \eqref{eq:Rinc}, we obtain
\[
 |R_{\rm inc}|
 <4\frac53(1.001)\varepsilon(n-1)^{-\varepsilon}
 \le\frac{4(5/3)(1.001)}{\e\log(n-1)}
 <\frac5{2\log(n-1)}<\frac5{26}.
\]
Here we used the fact that the maximum of $\varepsilon\e^{-\varepsilon\log(n-1)}$ is $1/(\e\log(n-1))$. Moreover, \eqref{eq:Flarge} gives $|F(1)|<10^{-4}$. Consequently,
\[
 |F(1)|+|R|<0.292+\frac5{26}+10^{-4}<\frac12\le1-a.
\]

Now suppose that $a\ge1/2$. Then $1/10\le1-a^2\le3/4$ and $x>a/2\ge1/4$, so \eqref{eq:Rdec} gives $|R_{\rm dec}|<112/(n-1)$. For $r\in[1/10,3/4]$, consider the function
\[
 g(r)=(1-r)\left(\frac{(n-1)r}{2}\right)^{-\sigma/r}.
\]
It is increasing on this interval, since
\[
 \frac{g'(r)}{g(r)}
 =-\frac1{1-r}
 +\frac{\sigma}{r^2}\left\{\log\frac{(n-1)r}{2}-1\right\}>0.
\]
Using $\Lambda\ge(n-1)(1-a^2)/2$ in \eqref{eq:Rinc} and writing $a^2=1-(1-a^2)$, we obtain
\begin{equation}\label{eq:Rincsmall}
 |R_{\rm inc}|
 \le\frac5{12}(n-1)
       \left(\frac{3(n-1)}8\right)^{-4\sigma/3}
 <0.016.
\end{equation}
For completeness,
\[
 \frac{4\sigma}{3}>\frac{33333}{25000}=\frac43-\frac1{75000},
\]
and the function $r(3r/8)^{-33333/25000}$ is decreasing for $r\ge10^6$. At $r=10^6$, the inequalities
\[
 72^3<375000,\qquad 375000^{1/75000}<1.001
\]
bound it by $1.001/27$, which proves \eqref{eq:Rincsmall}.

Similarly, with $k=(n-1)/(n+1)$, the function $((n-1)r/2)^{-k/r}$ is increasing on $[1/10,3/4]$. It follows from \eqref{eq:Flarge} that
\[
 |F(1)|\le\left(\frac{3(n-1)}8\right)^{-4k/3}
 <\frac8{3(n-1)}.
\]
Combining these estimates, we obtain
\[
 |F(1)|+|R|
 <\frac{112}{n-1}+\frac8{3(n-1)}+0.016
 <0.017<\frac1{20}<1-\sqrt{9/10}\le1-a.
\]

Thus $|F(1)|+|R|<1-a$ in both cases. In particular,
$|1-F(1)-R|>a$. Combining the first origin identity~\eqref{eq:origin} with the differential origin identity~\eqref{eq:derivative}, we obtain
\[
 |J|=\frac{n}{(n-1)a|x+iy|}|1-F(1)-R|
 >\frac{n}{(n-1)|x+iy|}>1,
\]
contradicting $|J|\le1$.
\end{proof}

\section{Finite exact verification}\label{sec:finite}

It remains to treat the degrees $6\le n\le10^6$. Although the degree range is finite, the distinguished zero $a$ still varies over the continuum $0<a<1$. The purpose of this section is to reduce this continuous parameter problem to finitely many exact rational comparisons.

The basic idea is as follows. We cover the parameter region by finitely many rectangles
\[
 n_-\le n\le n_+,\qquad a_-\le a\le a_+.
\]
On each rectangle, we first obtain a uniform lower bound for $1-\beta(1)$. Monotonicity then replaces all quantities appearing in the scalar inequalities of Proposition~\ref{prop:scalar} by explicit upper bounds depending only on the endpoints of the rectangle. The remaining integrals are bounded by convexity: on each integration subinterval, the integrand lies below its endpoint chord, so the integral is bounded by a rational linear combination of endpoint values. Consequently, excluding an entire rectangle reduces to checking one of finitely many explicit inequalities.

Thus the certificate is not a collection of sample points. Each record certifies a whole parameter rectangle, and the union of these rectangles covers the entire range $6\le n\le10^6$ and $0\le a\le1$.

\subsection{Bounds on a parameter rectangle}

Fix integers $6\le n_-\le n_+\le10^6$ and real numbers $0\le a_-<a_+\le1$. Consider the parameter rectangle
\begin{equation}\label{eq:box}
 n_-\le n\le n_+,\qquad a_-\le a\le a_+,
 \qquad 0<a<1.
\end{equation}
The first step is to obtain a lower bound for $1-\beta(1)$ that is valid uniformly throughout the rectangle. We denote such a bound by $\eta_0\in[0,1)$.

\begin{lem}[A uniform gap for $\beta(1)$]\label{lem:boxpolar}
Every configuration with $s\le1$ satisfying \eqref{eq:box} has
\[
 1-\beta(1)>\eta_0
\]
provided that either
\begin{equation}\label{eq:boxrational}
 \eta_0\le
 \frac1{1+\frac{n_+-1}{2}(1-a_-^2)}
\end{equation}
or
\begin{equation}\label{eq:boxchord}
 \left(1+(1-a_-^2)\eta_0\right)^{(n_++1)/2}
 -a_-^{n_++1}
 -\frac{n_-+1}{2}(1-a_+^2)(1+\eta_0)<0.
\end{equation}
\end{lem}

\begin{proof}
Suppose first that \eqref{eq:boxrational} holds. Since
\[
 \alpha\le\frac{n_+-1}{2}(1-a_-^2),
\]
the bound \eqref{eq:betabound} gives
\[
 1-\beta(1)>\frac1{1+\alpha}
 \ge
 \frac1{1+\frac{n_+-1}{2}(1-a_-^2)}
 \ge\eta_0.
\]

Now suppose that \eqref{eq:boxchord} holds. Fix the actual values of $n$ and $a$, and set
\[
 \Phi(r)=\left(1+(1-a^2)r\right)^{(n+1)/2}
 -a^{n+1}
 -\frac{n+1}{2}(1-a^2)(1+r).
\]
The left-hand side of \eqref{eq:boxchord} is an upper bound for $\Phi(\eta_0)$ throughout the rectangle. Hence $\Phi(\eta_0)<0$. Equivalently,
\[
 \int_0^1
 \left(a^2+(1-a^2)(1+\eta_0)t\right)^{(n-1)/2}
 \dd t<1.
\]
For fixed $n$ and $a$, the integral on the left is strictly increasing in $\eta_0$. On the other hand, \eqref{eq:chordpolar} says that the same integral, with $\eta_0$ replaced by $\rho$, is at least one. Therefore $\rho>\eta_0$. Since $\rho\le1-\beta(1)$ by Lemma~\ref{lem:gap}, we conclude that
\[
 1-\beta(1)>\eta_0.
\]
\end{proof}

Fix such a value of $\eta_0$ and put
\[
 d:=\frac{a_+^2+\eta_0}{2},\qquad
 r_-:=\frac{n_--2}{2},\qquad
 b(t):=1-\eta_0t-a_-^2t(1-t).
\]
Every rectangle appearing in the certificate satisfies $d<1$. Define
\begin{equation}\label{eq:Kdef}
 K_1:=\int_0^1t b(t)^{r_-}\dd t,
 \qquad
 K_2:=\int_0^1\frac{t^2b(t)^{r_-}}{1-d\,t}\dd t.
\end{equation}
These integrals are well defined on every certified rectangle. Indeed,
\[
 b(t)=(1-t)^2+(1-\eta_0)t+(1-a_-^2)t(1-t)>0,
 \qquad 0\le t\le1,
\]
and $b(t)\le1$.

The point of $b(t)$ is that it provides a single majorant for the functions $\beta(t)$ arising from all feasible configurations in the rectangle. This allows the scalar inequalities of Proposition~\ref{prop:scalar} to be replaced by bounds depending only on the rectangle endpoints.

\begin{lem}[Uniform bounds on a parameter rectangle]\label{lem:boxbounds}
On every feasible rectangle \eqref{eq:box}, the quantities
\[
 \begin{aligned}
 \overline F&=(1-\eta_0)^{(n_--1)/2},&
 \overline R&=\frac53a_+^2(n_+-1)K_1,\\
 \overline E&=(1-\eta_0)^{n_-/2},&
 \overline C&=\frac{5a_+^3n_+(n_+-1)}6K_2
 \end{aligned}
\]
are upper bounds for, respectively,
\[
\begin{gathered}
 \beta(1)^{(n-1)/2},\qquad
 \frac53a^2(n-1)\int_0^1t\beta(t)^{(n-2)/2}\dd t,\\
 \beta(1)^{n/2},\qquad C_n(a,x).
\end{gathered}
\]
\end{lem}

\begin{proof}
Fix $a$ and write $r=(n-2)/2\ge2$. Let
$\eta=1-\beta(1)$ and define
\[
 B_\eta(t):=1-\eta t-a^2t(1-t),\qquad
 D_\eta(t):=1-\frac{a^2+\eta}{2}\,t.
\]
Thus $B_\eta(t)=\beta(t)$ and $D_\eta(t)=1-axt$. As $\eta$ varies from $\eta_0$ to its actual value, both quantities remain positive. Moreover,
\[
 B_\eta(t)\le2D_\eta(t).
\]
Hence
\[
 \frac{\partial}{\partial\eta}
 \log\frac{B_\eta(t)^r}{D_\eta(t)}
 =-\frac{rt}{B_\eta(t)}+\frac{t}{2D_\eta(t)}
 \le0.
\]
It follows that replacing the actual value of $\eta$ by $\eta_0$ can only increase the ratio $B_\eta(t)^r/D_\eta(t)$. We may then replace $a^2$ by $a_-^2$ in the numerator and by $a_+^2$ in the denominator. The resulting numerator is $b(t)$. Since $0\le b(t)\le1$, replacing the exponent $r$ by the smaller value $r_-=(n_--2)/2$ enlarges the bound. We therefore obtain
\[
 \frac{\beta(t)^{(n-2)/2}}{1-axt}
 \le
 \frac{b(t)^{r_-}}{1-d\,t}.
\]
This proves the estimate for $\overline C$.

The same argument without the denominator gives
\[
 \beta(t)^{(n-2)/2}\le b(t)^{r_-},
\]
and hence the bound for $\overline R$. Finally, the endpoint bounds for $\overline F$ and $\overline E$ follow directly from
$1-\beta(1)>\eta_0$ and $n\ge n_-$.
\end{proof}

At this point every term in the two scalar inequalities of Proposition~\ref{prop:scalar} has been replaced by a quantity depending only on the rectangle. Thus, to exclude the whole rectangle, it is enough to show that one of the corresponding uniform upper bounds is strictly less than one.

\begin{prop}[Rectangle exclusion tests]\label{prop:boxtests}
A rectangle \eqref{eq:box} contains no configuration with $s\le1$ if any one of the following conditions holds:
\begin{align}
 &\overline F+\overline R+\frac{(n_+-1)a_+}{n_+}<1;
 \label{eq:testdirect}\\
 &2\overline C\le a_+,\qquad a_++\overline E<1;
 \label{eq:testmonotone}\\
 &2\overline C>a_+,\quad\overline E<1,\quad
 4(a_++\overline C)^3<27\overline C(1-\overline E)^2;
 \label{eq:testcubic}\\
 &2\overline C\le a_-,\quad
 0<\alpha_+\le\frac{n_--2}{2},\quad
 \frac{n_+-1}{2\alpha_+}
 \left(\frac{\alpha_+}{1+\alpha_+}\right)^{n_-/2}<\frac14,
 \label{eq:testboundary}
\end{align}
where
\[
 \alpha_+:=\frac{n_+-1}{2}(1-a_-^2).
\]
\end{prop}

\begin{proof}
Suppose first that \eqref{eq:testdirect} holds. By Lemma~\ref{lem:boxbounds}, the right-hand side of \eqref{eq:directscalar} is bounded above by
\[
 \overline F+\overline R+\frac{(n_+-1)a_+}{n_+},
\]
which is strictly less than one. This contradicts \eqref{eq:directscalar}.

We next consider the centered inequality \eqref{eq:centerscalar}. By \eqref{eq:cubicmax}, its first two terms are bounded by
$\Psi(a,C_n(a,x))$. Since $\Psi$ is nondecreasing in both variables,
\[
 \Psi(a,C_n(a,x))
 \le\Psi(a_+,\overline C).
\]
If \eqref{eq:testmonotone} holds, then
$2\overline C\le a_+$, and therefore
$\Psi(a_+,\overline C)=a_+$. Hence
\[
 \Psi(a_+,\overline C)+\overline E<1.
\]
If instead \eqref{eq:testcubic} holds, then
\[
 \Psi(a_+,\overline C)
 =\frac{2(a_++\overline C)^{3/2}}
 {3\sqrt{3\overline C}},
\]
and the inequality in \eqref{eq:testcubic} is exactly the squared form of
\[
 \Psi(a_+,\overline C)<1-\overline E.
\]
Thus either condition makes the right-hand side of
\eqref{eq:centerscalar} strictly less than one, again a contradiction.

Finally, suppose that \eqref{eq:testboundary} holds. Since
$2\overline C\le a_-$, we have
\[
 C_n(a,x)\le\frac a2
\]
throughout the rectangle. From \eqref{eq:betabound},
\[
 \frac{\beta(1)^{n/2}}{1-a^2}
 \le\frac{n_+-1}{2}
 \frac{\alpha^{n_-/2-1}}{(1+\alpha)^{n_-/2}}.
\]
The function
$\alpha^{n_-/2-1}/(1+\alpha)^{n_-/2}$
is increasing for
$0<\alpha\le(n_--2)/2$. Since $\alpha\le\alpha_+$,
condition \eqref{eq:testboundary} gives
\[
 \beta(1)^{n/2}<\frac{1-a^2}{4}.
\]
Because $C_n(a,x)\le a/2$, formula \eqref{eq:cubicmax} gives
\[
 \Psi(a,C_n(a,x))=a.
\]
Consequently, \eqref{eq:centerscalar} would imply
\[
 1<a+\frac{1-a^2}{4},
\]
whereas
\[
 a+\frac{1-a^2}{4}<1
\]
for $0<a<1$. This contradiction completes the proof.
\end{proof}

\subsection{Rigorous integration}

It remains to bound the two integrals in \eqref{eq:Kdef}. No numerical quadrature approximation is needed. Instead, convexity gives an explicit upper bound on each integration subinterval in terms of the values at its two endpoints. Since the subdivision points are rational, this reduces the integral bounds to exact rational arithmetic, apart from the controlled enclosure of powers.

\begin{lem}[Convex quadrature bounds]\label{lem:quadrature}
The functions $b(t)^{r_-}$ and $b(t)^{r_-}/(1-d\,t)$ are convex on $[0,1]$. For any convex function $f$ on $[l,l+z]$, with $z>0$ and $l\ge0$,
\begin{equation}\label{eq:weights}
 \begin{aligned}
 \int_l^{l+z}tf(t)\dd t
 &\le z\left(\frac l2+\frac z6\right)f(l)
      +z\left(\frac l2+\frac z3\right)f(l+z),\\
 \int_l^{l+z}t^2f(t)\dd t
 &\le z\left(\frac{l^2}2+\frac{lz}3+\frac{z^2}{12}\right)f(l)\\
 &\quad+z\left(\frac{l^2}2+\frac{2lz}3+\frac{z^2}4\right)f(l+z).
 \end{aligned}
\end{equation}
\end{lem}

\begin{proof}
The function $b$ is a nonnegative convex quadratic. Since $r_-\ge2$, the function $b^{r_-}$ is convex.

For $q\ge2$, consider
\[
 (B,D)\longmapsto\frac{B^q}{D},
 \qquad B\ge0,\quad D>0.
\]
This function is convex and nondecreasing in $B$. For $B>0$, convexity follows from the nonnegative diagonal entries of its Hessian and the determinant
\[
 q(q-2)B^{2q-2}/D^4\ge0;
\]
the case $B=0$ follows by continuity. Since $b(t)$ is convex and
$1-d\,t$ is positive and affine, it follows that
$b(t)^{r_-}/(1-d\,t)$ is convex on $[0,1]$.

Now let $f$ be convex on $[l,l+z]$. Its graph lies below the chord joining the two endpoint values. Thus
\[
 f(t)\le
 \frac{l+z-t}{z}f(l)
 +\frac{t-l}{z}f(l+z).
\]
Multiplying by $t$ and $t^2$, respectively, and integrating over
$[l,l+z]$ gives \eqref{eq:weights}.
\end{proof}

Thus each of $K_1$ and $K_2$ is bounded above by a finite sum involving only endpoint values of the corresponding convex integrand. Once the subdivision points are chosen on a dyadic grid, all coefficients in these sums are rational. The remaining task is therefore a finite exact verification.

\subsection{The certificate and its verification}

Let $S=2^{100}$. The JSON certificate uses the legacy field names
\begin{center}\small
\code{[m_lower,m_upper,a_lower_integer,a_upper_integer,h_integer,test,subdivision]}.
\end{center}
In the notation of this paper, these fields encode
\[
 n_-=\code{m_lower}+1,\qquad n_+=\code{m_upper}+1,
\]
while $\code{h_integer}/S$ represents the quantity $\eta_0$. The endpoints of the $a$-interval are
$a_-=\code{a_lower_integer}/S$ and
$a_+=\code{a_upper_integer}/S$. We also write
$v_0=\code{subdivision}$. Thus the code retains its original interface, whereas the paper uses the degree $n$ throughout.

For a record with upper degree $n_+$, define
\[
 K=\left\lfloor\log_2(n_+-1)\right\rfloor+4.
\]
We begin with the integer breakpoints
\[
 0,\ S\mathbin{\gg}K,\ S\mathbin{\gg}(K-1),\ldots,
 S\mathbin{\gg}1,\ S,
\]
where
$S\mathbin{\gg}k=\lfloor S/2^k\rfloor$.
On each consecutive interval $[L,U]$, we insert the points
$L+\lfloor j(U-L)/v_0\rfloor$ for
$0\le j\le v_0$, add their reflections $S-t$, sort the resulting set, and divide by $S$.

All subdivision points are therefore rational. Since the coefficients in \eqref{eq:weights} are rational functions of these points, the corresponding quadrature weights are rational as well. Thus the passage from the continuous integrals in \eqref{eq:Kdef} to the quantities checked by the certificate introduces no floating-point quadrature error.

The first checker encloses every arithmetic operation on a $100$-bit dyadic grid. The second checker computes the coordinates, weights, and rational coefficients exactly, and uses a $160$-bit dyadic grid only for the enclosure of powers. The only nonintegral exponents that occur are half-integers.

The finite verification required for the proof is summarized in the following proposition.

\begin{prop}[Finite verification]\label{prop:finite}
For every integer $6\le n\le10^6$ and every $0<a<1$, the inequality $s\le1$ is impossible.
\end{prop}

\begin{proof}
The fixed public release
\href{\releaseURL}{\texttt{v1.0}}
contains the certificate
\href{https://github.com/zhangteng2000/quadratic-tang-zhang-conjecture/blob/v1.0/supplement/certificate.json}{\code{certificate.json}}.
It consists of $6593$ records arranged in $430$ consecutive degree blocks.

For each record, both
\href{https://github.com/zhangteng2000/quadratic-tang-zhang-conjecture/blob/v1.0/supplement/verify_certificate.py}{\code{verify_certificate.py}}
and
\href{https://github.com/zhangteng2000/quadratic-tang-zhang-conjecture/blob/v1.0/supplement/verify_independent.py}{\code{verify_independent.py}}
verify the admissible ranges of the rectangle endpoints, one of the two lower-bound conditions
\eqref{eq:boxrational} and \eqref{eq:boxchord}, the positivity of $1-d$, the convex integral bounds obtained from \eqref{eq:weights}, and one of the four exclusion conditions in Proposition~\ref{prop:boxtests}. Thus each accepted record proves that an entire parameter rectangle contains no configuration with $s\le1$.

The records are distributed among the four exclusion tests as follows:
\begin{center}
\begin{tabular}{lr}
\toprule
Exclusion condition & Number of records\\
\midrule
\eqref{eq:testdirect} & 4003\\
\eqref{eq:testmonotone} & 5\\
\eqref{eq:testcubic} & 2054\\
\eqref{eq:testboundary} & 531\\
\midrule
Total & 6593\\
\bottomrule
\end{tabular}
\end{center}

Figure~\ref{fig:certificate-coverage} gives a geometric view of the
finite certificate.  Each colored region represents an entire certified
parameter rectangle, rather than a collection of sampled parameter values.
The coloring indicates which of the four exclusion tests is used to certify
the corresponding region.

\begin{figure}[tbp]
	\centering
	\includegraphics[width=0.8\textwidth]{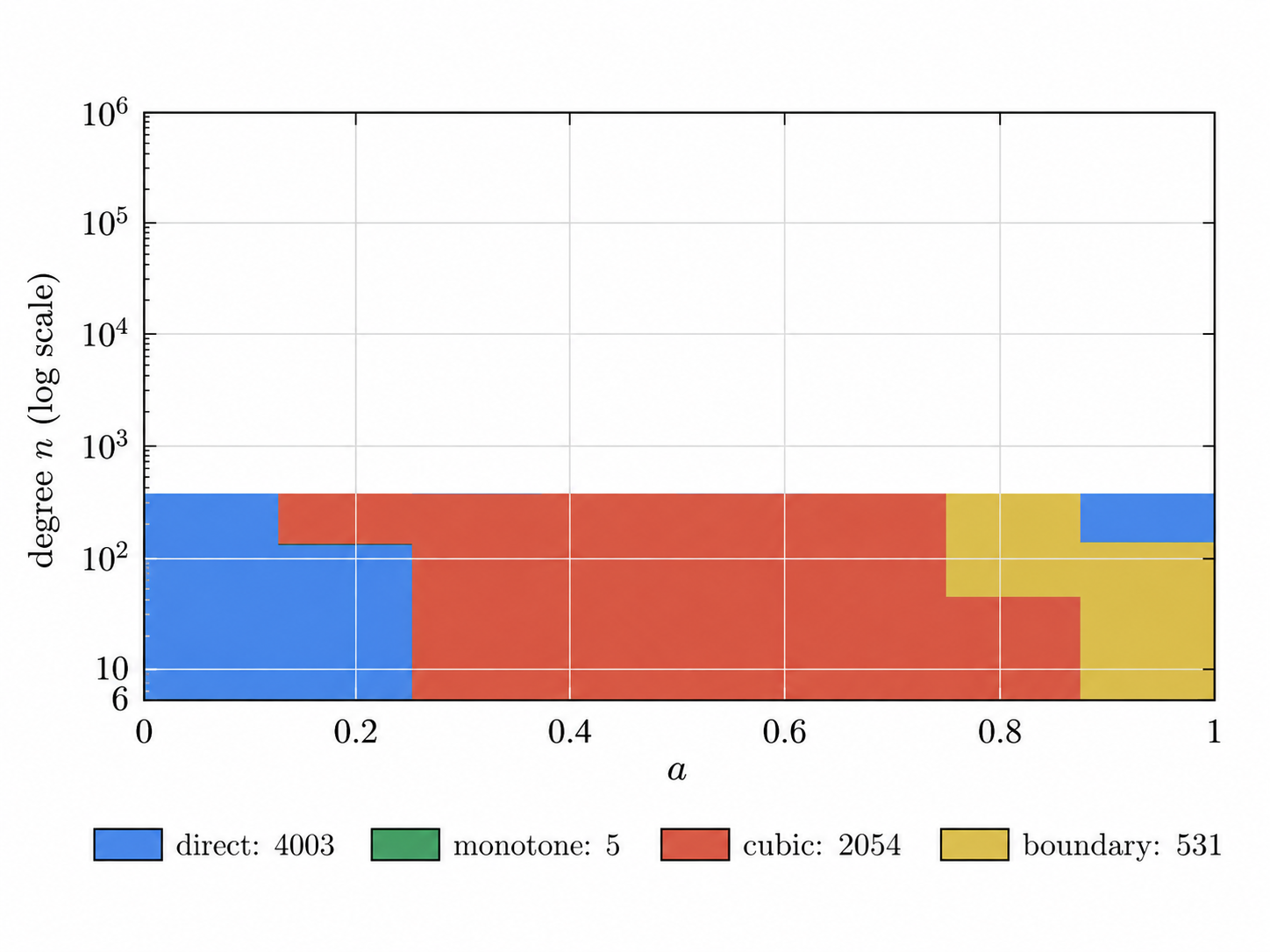}
	\caption{
		Coverage of the finite certificate in the $(a,n)$ parameter plane.
		The colors indicate the exclusion condition used for the certified
		rectangles: blue for the direct test, green for the monotonicity test,
		red for the cubic test, and gold for the boundary test.  The vertical
		degree axis is displayed on a logarithmic scale, from $n=6$ to
		$n=10^6$.  Thus the figure illustrates the coverage of the admissible
		parameter region by certified rectangles, rather than verification at
		isolated sample points.
	}
	\label{fig:certificate-coverage}
\end{figure}

Among the lower bounds for $1-\beta(1)$, $513$ records use
\eqref{eq:boxrational}, while the remaining $6080$ use
\eqref{eq:boxchord}. The stored range from
\code{m_lower}=5 to \code{m_upper}=999999 corresponds exactly to the degree range
$6\le n\le10^6$.

Within each degree block, the certified intervals in $a$
cover $[0,1]$ without gaps and have pairwise disjoint interiors. Hence every pair $(n,a)$ in the required range belongs to at least one certified parameter rectangle. Since every such rectangle is excluded by Proposition~\ref{prop:boxtests}, no configuration with $s\le1$ can occur for
$6\le n\le10^6$ and $0<a<1$.

Both verification programs return \code{PASS} on the supplied certificate. The certificate is identified by the SHA--256 digest
\begin{center}\small\ttfamily
8384d62a377b4fcac4080d0e05979d7a6fcf\\
90437b7e4911927a5403d4052709.
\end{center}
The line break and final punctuation are not part of the digest.
\end{proof}

The role of the computation is therefore limited to checking a finite family of explicit inequalities. The reduction from the original continuous parameter problem to these inequalities is entirely contained in Lemmas~\ref{lem:boxpolar}, \ref{lem:boxbounds}, and \ref{lem:quadrature}, together with Proposition~\ref{prop:boxtests}.

\section{Endpoints, small degrees, and equality}\label{sec:completion}

The preceding arguments treat every interior zero when $n\ge6$. We now dispose of the remaining small degrees and the two endpoint cases. No finite verification is needed for $2\le n\le5$. The polar argument used below is closely related to the small-degree arguments in \cite[p. 16]{Maz26} and \cite[Section~4, Remark~16]{Tao26}.

\begin{lem}[Small degrees]\label{lem:small}
The inequality $s\le1$ is impossible when $2\le n\le5$ and $0<a<1$.
\end{lem}

\begin{proof}
Since $x\le1$ and $s\le1$, the integrand in \eqref{eq:polarraw} is bounded above by
\[
 \bigl(a+(1-a^2)t\bigr)^{n-1}.
\]
For the fourth power, direct integration and expansion give
\[
 1-\int_0^1\bigl(a+(1-a^2)t\bigr)^4\dd t
 =\frac{(1-a)^3(1+a)}5\bigl(a^4-3a^3+3a+4\bigr)>0.
\]
Indeed, the final factor may be written as $a^4+3a(1-a^2)+4>0$. Since $n-1\le4$, monotonicity of the integral power means on the probability space $[0,1]$ gives
\[
 \int_0^1\bigl(a+(1-a^2)t\bigr)^{n-1}\dd t
 \le\left(\int_0^1\bigl(a+(1-a^2)t\bigr)^4\dd t\right)^{(n-1)/4}<1.
\]
This contradicts \eqref{eq:polarraw}.
\end{proof}

At $a=1$ we shall use the following boundary inequality,
obtained by setting $k=\nu=1$ in
\cite[p.~460, Eq.~(2.3)]{MS69}.

\begin{lem}[Meir--Sharma boundary inequality]\label{lem:MSboundary}
Suppose that all zeros of $p$ lie in the closed unit disk, that $p(1)=0$, and that $1$ is a simple zero. If $\zeta_1,\ldots,\zeta_{n-1}$ are the critical points of $p$, counted with multiplicity, then
\begin{equation*}
 \frac1{n-1}\sum_{j=1}^{n-1}
 \Ree\frac1{1-\zeta_j}\ge1.
\end{equation*}
\end{lem}

We can now treat the two endpoints.

\begin{lem}[Endpoint cases]\label{lem:endpoints}
Suppose that the prescribed zero is not critical. At $a=0$,
\[
 s\ge n^{2/(n-1)}>1.
\]
At $a=1$, one has $s\ge1$, with equality if and only if $p$ is a nonzero constant multiple of $z^n-1$.
\end{lem}

\begin{proof}
Suppose first that $a=0$. Comparison of constant terms gives
\[
 \frac{n}{\prod_{j=1}^{n-1}q_j}
 =p'(0)=\prod_{j=1}^{n-1}(-z_j).
\]
Since $|z_j|\le1$,
\[
 \prod_{j=1}^{n-1}|q_j|\ge n.
\]
The arithmetic--geometric mean inequality therefore gives
\[
 s\ge
 \left(\prod_{j=1}^{n-1}|q_j|\right)^{2/(n-1)}
 \ge n^{2/(n-1)}>1.
\]

Now suppose that $a=1$. Put $q_j=(1-\zeta_j)^{-1}$. By Lemma~\ref{lem:MSboundary},
\[
 1\le
 \frac1{n-1}\sum_{j=1}^{n-1}\Ree q_j.
\]
By Cauchy--Schwarz,
\[
 \frac1{n-1}\sum_{j=1}^{n-1}\Ree q_j
 \le
 \left(\frac1{n-1}\sum_{j=1}^{n-1}|q_j|^2\right)^{1/2}
 =\sqrt{s}.
\]
Hence $s\ge1$.

Suppose now that $s=1$. The preceding inequalities must then be equalities, and hence
\[
 \frac1{n-1}\sum_{j=1}^{n-1}\Ree q_j=1.
\]
Consequently,
\[
 \frac1{n-1}\sum_{j=1}^{n-1}|q_j-1|^2
 =s-2\frac1{n-1}\sum_{j=1}^{n-1}\Ree q_j+1=0.
\]
Thus $q_j=1$ for every $j$, and therefore $\zeta_j=0$ for every $j$. It follows that $p'$ is a nonzero constant multiple of $z^{n-1}$. Since $p(1)=0$, the polynomial $p$ is a nonzero constant multiple of $z^n-1$. Conversely, for such a polynomial every critical point is zero, and hence $s=1$.
\end{proof}

We are now ready to prove the main theorem.

\begin{proof}[Proof of Theorem~\ref{thm:main}]
If the prescribed zero is also a critical point, then the sum in \eqref{eq:main} is infinite, so the inequality is strict. We may therefore assume that the prescribed zero is simple and pperform the normalization of Section~\ref{sec:notation}.

For $0<a<1$, Lemma~\ref{lem:small}, Proposition~\ref{prop:finite}, and Propositions~\ref{prop:near}--\ref{prop:away} exclude $s\le1$ in every degree. Hence $s>1$, and therefore
\[
 \sum_{j=1}^{n-1}\frac1{|a-\zeta_j|^2}>n-1
\]
for every interior zero.

At $a=0$, Lemma~\ref{lem:endpoints} again gives strict inequality. At $a=1$, the same lemma gives $s\ge1$, with equality precisely when $p$ is a nonzero constant multiple of $z^n-1$. Undoing the rotation, the equality family becomes
\[
 p(z)=A(z^n-\omega),
 \qquad A\ne0,\qquad |\omega|=1.
\]
Conversely, every critical point of $A(z^n-\omega)$ is zero, while every zero of the polynomial has modulus one. Hence equality holds for every zero of every polynomial in this family.
\end{proof}

For completeness, we conclude with the proof of Corollary~\ref{cor:powers}.

\begin{proof}[Proof of Corollary~\ref{cor:powers}]
When all distances are nonzero, put $t_j=|a-\zeta_j|^{-1}$. For $\lambda\ge2$, monotonicity of power means and \eqref{eq:main} give
\[
 \left(\frac1{n-1}\sum_jt_j^\lambda\right)^{1/\lambda}
 \ge\left(\frac1{n-1}\sum_jt_j^2\right)^{1/2}\ge1.
\]
Thus
\[
 \sum_{j=1}^{n-1}|a-\zeta_j|^{-\lambda}\ge n-1.
\]
If equality holds, then both power means above are equal to one. In particular, the quadratic mean is one, so Theorem~\ref{thm:main} shows that
$p(z)=A(z^n-\omega)$ with $A\ne0$ and $|\omega|=1$. Conversely, this family gives equality for every $\lambda\ge2$. If one of the distances vanishes, the corresponding sum is infinite and equality cannot occur.
\end{proof}

\appendix
\section{Constants and reproducibility}\label{app:audit}

\subsection{Analytic comparisons}\label{app:constants}

This appendix records the numerical comparisons used in Section~\ref{sec:large}. They are also verified using exact rational arithmetic in \code{analytic_audit.py}. In particular, the large-degree argument does not rely on extrapolation from a finite collection of numerical checks.

The exponential series, with its tail bounded by a geometric series, gives
\[
\frac{27}{10}<\sum_{j=0}^5\frac1{j!}<\e
<\sum_{j=0}^5\frac1{j!}+\frac7{4320}<\frac{11}{4}.
\]
Hence
\[
13<\log10^6<14,\qquad
\log(3\cdot10^6/8)>12,\qquad
\log(10^6/20)>4,
\]
and $\sqrt\e<5/3$. For $n-1\ge10^6$,
\[
\frac{\sqrt{n-1}}6>4\log(n-1),\qquad
\frac{9(n-1)}{160}>4\log(n-1),
\]
and both inequalities become stronger as $n$ increases.

For the constants arising in the integral estimates, the functions
\[
r\longmapsto\frac{r^2(r+1)}{(r-1)^3},
\qquad
r\longmapsto\left(\frac r{r-1}\right)^2
\]
are decreasing for $r>1$. It therefore suffices to check them at $r=10^6$, where
\[
\frac{2048(5/3)}{81}\left(\frac{10}{9}\right)^3
\frac{(10^6)^2(10^6+1)}{(10^6-1)^3}<60,
\qquad
4\frac53\left(\frac{10^6}{10^6-1}\right)^2<7.
\]
The remaining comparisons in the near-boundary case reduce to
\[
10\left(1-\frac3{10^6}\right)>9,
\qquad
\frac{60}{10^6}+\frac7{10^{12}}<\frac{61}{10^6}<\frac9{20},
\]
together with $(n-1)2^{-n/2}<1/4$ for $n-1\ge16$.

The function $\log(n-1)/(n+1)$ is decreasing throughout the range under consideration. Using $\e^t\le(1-t)^{-1}$ for $0\le t<1$, we obtain
\[
(n-1)^{4/(n+1)}
<\frac1{1-56/(10^6+2)}<\frac{1001}{1000}.
\]
The elementary estimates needed in the case $0<a\le1/2$ are
\[
(8/3)^4<4^3,\qquad
\frac{4(5/3)(1001/1000)}{27/10}<\frac52,
\qquad
\frac{7(10^6+2)}{24\cdot10^6}<\frac{292}{1000},
\]
and
\[
\frac{292}{1000}+\frac5{26}+\frac1{10000}<\frac12.
\]

For the case $a\ge1/2$, we have $\sigma=(n-2)/(n+1)>99999/100000$ and $72^3<375000$. Moreover,
\[
375000^{1/75000}
<\frac1{1-14/75000}<\frac{1001}{1000},
\qquad
\frac5{12}\frac1{27}\frac{1001}{1000}<\frac{16}{1000}.
\]
Finally,
\[
\frac{112}{10^6}+\frac8{3\cdot10^6}+\frac{16}{1000}
<\frac{17}{1000}<\frac1{20},
\qquad (19/20)^2>9/10.
\]
All terminating decimals appearing in the proof are understood as exact rational numbers.

\subsection{Certificate format and directed rounding}\label{app:format}

The JSON certificate contains the fields \code{precision_bits}, \code{m_start},
\code{m_stop_exclusive}, and \code{records}, together with auxiliary generation statistics. The names involving \code{m} are retained from the released code, where \code{m} denotes $n-1$. The first three values are $100$, $5$, and $1000000$, respectively. Each record is a seven-entry array
\begin{center}\small
\code{[m_lower,m_upper,a_lower_integer,a_upper_integer,h_integer,test,subdivision]}.
\end{center}

The legacy field \code{h_integer} stores the lower bound $\eta_0$ for $1-\beta(1)$. The three coordinate integers are stored as decimal strings and interpreted after division by $2^{100}$. Every record uses \code{subdivision = 2}. The field \code{test} is only an annotation recording the exclusion test used during certificate generation; the second checker independently verifies a valid test rather than relying on this annotation. The additional condition $\eta_0>2a_+-a_+^2$ would also exclude a rectangle, but it is not needed for any record in the supplied certificate.

For an interval $[L/S,U/S]$ and another nonnegative interval $[L'/S,U'/S]$, multiplication is enclosed by
\[
\left[\frac{\lfloor LL'/S\rfloor}{S},
\frac{\lceil UU'/S\rceil}{S}\right].
\]
For a quotient with positive denominator interval, the enclosure is
\[
\left[\frac{\lfloor LS/U'\rfloor}{S},
\frac{\lceil US/L'\rceil}{S}\right].
\]
Square roots are bounded using integer square-root enclosures. Signed additions and subtractions use the corresponding endpoint order. Integer powers are computed by repeated squaring, while a half-integer power is evaluated as the product of an integer power and a square root.

The first checker applies these directed-rounding rules to every arithmetic operation. The second checker uses exact rational arithmetic except for the enclosure of powers, where it works on a $160$-bit dyadic grid. In the implementation, interval bounds for $b$ are replaced by their nonnegative parts; this is valid because $b\ge0$, as proved before Lemma~\ref{lem:boxbounds}.

\subsection{Reproduction and scope of verification}\label{app:reproduce}

The certificate and Python reproducibility package are available in the public
\href{\repoURL}{GitHub repository}; the version used for this manuscript is the fixed
\href{\releaseURL}{release \texttt{v1.0}}. The files needed to reproduce the arithmetic verification are:
\begin{itemize}[leftmargin=2em,itemsep=0.15em,topsep=0.3em]
\item \href{https://github.com/zhangteng2000/quadratic-tang-zhang-conjecture/blob/v1.0/supplement/certificate.json}{\code{certificate.json}} --- finite interval certificate;
\item \href{https://github.com/zhangteng2000/quadratic-tang-zhang-conjecture/blob/v1.0/supplement/certify.py}{\code{certify.py}} --- certificate generator;
\item \href{https://github.com/zhangteng2000/quadratic-tang-zhang-conjecture/blob/v1.0/supplement/verify_certificate.py}{\code{verify_certificate.py}} --- first checker;
\item \href{https://github.com/zhangteng2000/quadratic-tang-zhang-conjecture/blob/v1.0/supplement/verify_independent.py}{\code{verify_independent.py}} --- separately implemented checker;
\item \href{https://github.com/zhangteng2000/quadratic-tang-zhang-conjecture/blob/v1.0/supplement/analytic_audit.py}{\code{analytic_audit.py}} --- exact audit of the analytic constants.
\end{itemize}

From the supplementary directory, the following commands require only Python's standard library:
\begin{verbatim}
python verify_certificate.py certificate.json
python verify_independent.py certificate.json
python analytic_audit.py
\end{verbatim}
The first two commands verify Proposition~\ref{prop:finite}; the third checks the numerical comparisons recorded in Appendix~\ref{app:constants}. The certificate can be regenerated by running
\begin{verbatim}
python certify.py --stop 1000000 --out regenerated_certificate.json
python verify_independent.py regenerated_certificate.json
\end{verbatim}
Floating-point arithmetic in the generator is used only to propose search brackets and report progress; it never determines whether a rectangle is accepted.

The first checker uses the interval-arithmetic and exclusion-test routines in \code{certify.py}.
The second checker does not import \code{certify.py}
or \code{verify_certificate.py}. The two verification programs are separate implementations of the same mathematical specification. Their agreement provides a cross-check of the arithmetic specification. The completed Lean~4 formalization~\cite{Zha26b} also proves the analytic reductions, verifier soundness, rectangle coverage, acceptance of all $6593$ certificate records, and the final theorem, including the equality classification and Corollary~\ref{cor:powers}.

\section*{Data availability}
The finite certificate, generator, two checkers, analytic constant checker, and verification logs are publicly available in the
\href{\repoURL}{GitHub repository} and archived in the fixed
\href{\releaseURL}{release \texttt{v1.0}}. The Lean~4 source and verification records are available in the separate formalization repository~\cite{Zha26b}. The certificate is an essential component of the computer-assisted proof. No empirical data are used.

\end{document}